\documentclass[a4paper,11pt]{article}
\usepackage{fullpage}

\usepackage{hyperref}
\hypersetup{
    colorlinks,
    citecolor=green,
    filecolor=black,
    linkcolor=blue,
    urlcolor=blue
}
\hypersetup{linktocpage}

\hypersetup{linktocpage}

\usepackage{cite}
\usepackage{graphicx}

\usepackage{amsbsy}
\usepackage{latexsym}
\usepackage{amsfonts}
\usepackage{amssymb}
\usepackage[usenames]{color}
\usepackage{amsmath,amsthm}
\usepackage{enumerate}
\usepackage{tikz}
\usetikzlibrary{arrows}
\usepackage{mathdots}
\usepackage{mathtools}

\usepackage{stmaryrd}
\usepackage{dsfont}
\usepackage{color}

\usepackage{authblk}
\usepackage{graphics}
\usepackage{graphicx}
\usepackage{stmaryrd}
\usepackage{dsfont}

\DeclareMathOperator{\linspan}{span}

\newtheorem{theorem}{Theorem}[section]
\newtheorem{lemma}[theorem]{Lemma}

\newtheorem{cor}[theorem]{Corollary}

\theoremstyle{definition}

\newtheorem{remark}[theorem]{Remark}

\newcommand{\Chi}{\raise .3ex
\hbox{\large $\chi$}}

\newcommand{\R}{\mathbb{R}}

\newcommand{\N}{\mathbb{N}}

\newcommand\dist{\mathop{\rm dist}}

\usepackage[section]{algorithm}
\usepackage{algorithmicx}
\usepackage{algpseudocode}
\algrenewcommand\algorithmicrequire{\makebox[46pt][l]{\textrm{required:}}}
\algrenewcommand\algorithmicensure{\makebox[46pt][l]{\textrm{output:}}}
\algrenewcommand\algorithmicfunction{\textrm{function}}
\algrenewcommand\algorithmicwhile{\textrm{while}}
\algrenewcommand\algorithmicdo{}
\algrenewcommand\algorithmicend{\textrm{end}}
\algrenewcommand\algorithmicforall{\textrm{for all}}
\algrenewcommand\algorithmicfor{\textrm{for}}
\algrenewcommand\algorithmicrepeat{\textrm{repeat}}
\algrenewcommand\algorithmicuntil{\textrm{until}}
\algrenewcommand\algorithmicif{\textrm{if}}
\algrenewcommand\algorithmicthen{\textrm{then}}
\algrenewcommand\algorithmicelse{\textrm{else}}

\usepackage{todonotes}

\newcommand{\be}{\begin{equation}}
\newcommand{\ee}{\end{equation}}
\newcommand{\beq}{\begin{eqnarray}}
\newcommand{\beqq}{\begin{eqnarray*}}
\newcommand{\eeq}{\end{eqnarray}}
\newcommand{\eeqq}{\end{eqnarray*}}

\numberwithin{equation}{section}

\title{Greedy Algorithms and Kolmogorov Widths in Banach Spaces}

\author[a,b]{Van Kien Nguyen\thanks{E-mail: vnguyen@ins.uni-bonn.de,\ kiennv@utc.edu.vn}}

\date{\today}

\begin{document}
 
\maketitle
\begin{abstract}  Let $X$ be a Banach space and $\mathcal{K}$ be a compact subset in $X$. We consider a greedy algorithm for finding an $n$-dimensional subspace $V_n\subset X$ which can be used to approximate the elements of $\mathcal{K}$. We are interested in how well the space $V_n$ approximates the elements of $\mathcal{K}$. For this purpose we compare the performance of greedy algorithm measured by $\sigma_n(\mathcal{K})_X:=\dist(\mathcal{K},V_n)_X$ with the Kolmogorov width $d_n(\mathcal{K})_X$ which is the best possible error one can achieve when approximating $\mathcal{K}$ by $n$-dimensional subspaces. Various results in this direction have been given, e.g., in Binev et al. (SIAM J. Math. Anal. (2011)), DeVore et al. (Constr. Approx. (2013)) and Wojtaszczyk (J. Math. Anal. Appl. (2015)). The purpose of the present paper is to continue this line. We shall show that there exists a constant $C>0$ such that 
\beqq 
\sigma_n(\mathcal{K})_X\leq  C n^{-s+\mu}\big(\log(n+2)\big)^{\min(s,1/2)}, \quad  \ n\geq 1\,, 
\eeqq	
if Kolmogorov widths $d_n(\mathcal{K})_X$ decay  as $n^{-s}$ and the Banach-Mazur distance between an arbitrary $n$-dimensional subspace  $V_n \subset X$ and $\ell_2^n$ satisfies $d(V_n,\ell_2^n)\leq C_1 n^\mu$. In particular, when some additional information about the set $\mathcal{K}$ is given then there is no logarithmic factor in this estimate.   
\end{abstract}
{\bf \qquad Key words:} Greedy algorithms, Kolmogorov widths, compactness
\section{Introduction}
Recently, a new greedy algorithm for obtaining a good subspace $V_n$ of $n$-dimension to  approximate elements of a compact set $\mathcal{K}$ in a Banach space $X$ has been given. This greedy algorithm was studied initially when $X$ is a Hilbert space in the context of reduced basis methods for solving  families of PDEs, see \cite{May1,May2}. Later, it was studied extensively not only in the setting of Hilbert spaces, let us mention, for instance, Binev et al. \cite{Biet}, Buffa et al. \cite{Buet}, DeVore et al. \cite{DePeVo}, and Wojtaszczyk \cite{Woj}.   The greedy algorithm for generating the subspace $V_n$ to approximate elements of $\mathcal{K}$ is implemented as follows. We first select $f_0$ such that
\beqq
\|f_0\|_X =\max_{f\in \mathcal{K}}\|f\|_X \,.
\eeqq
Since $\mathcal{K}$ is compact, such a $f_0$ always exists. At the general step, assuming that $\{f_0,\ldots,f_{n-1} \}$ and $V_{n}=\linspan\{f_0,\ldots,f_{n-1} \}$ have been chosen, then we take $f_n$ such that 
\beqq
\dist (f_n,V_{n})_X=\max_{f\in \mathcal{K}}\dist(f,V_{n})_X\,.
\eeqq
The error in approximating the elements of $\mathcal{K}$ by $V_n$ is defined as 
\be \label{dis}
\sigma_0(\mathcal{K})_X:= \|f_0\|_X\,\qquad\text{and}\qquad \sigma_n(\mathcal{K})_X:=\dist (f_n,V_{n})_X
\ee
for $n\geq 1$. The sequence $ \{\sigma_n(\mathcal{K})_X \}_{n\geq 0}$ is monotonically non-increasing. It is important to note that the sequence $\{f_n\}_{n\geq 0}$ and also $ \{\sigma_n(\mathcal{K})_X \}_{n\geq 0}$ are not unique.

\smallskip
 Let us mention that the best possible error one can achieve when approximating the elements of $\mathcal{K}$ by $n$-dimensional subspaces is the Kolmogorov width $d_n(\mathcal{K})_X$, which is given by
\beqq
d_n(\mathcal{K})_X:= \inf_{L}\sup_{f\in \mathcal{K}}\dist(f,L)_X\,, \qquad n\geq 1,
\eeqq
where the infimum is taken over all $n$-dimensional subspaces $L$ of $X$. We also put $$d_0(\mathcal{K})_X=\max_{f\in \mathcal{K}}\|f\|_X.$$ We would like to emphasize that in practice, finding subspaces which give this performance is out of reach.

\smallskip
 We are interested in how well the subspaces created by the greedy algorithm approximate the elements of $\mathcal{K}$. For this purpose it is natural to compare $\sigma_n(\mathcal{K})_X$ with the Kolmogorov width $d_n(\mathcal{K})_{X}$. Various comparisons between $\sigma_n(\mathcal{K})_X$ and $d_n(\mathcal{K})_X$ have been made. The first attempt in this direction was given in \cite{Buet}, where the authors considered the case when $X$ is a Hilbert space $H$. Under this assumption, it has been shown that
\beqq
\sigma_n(\mathcal{K})_H \leq C2^nd_n(\mathcal{K})_H
\eeqq
for an absolute constant $C$. Observe that this result is useful only   when $d_n(\mathcal{K})_H$ decays faster than $2^{-n}$. A significant improvement of the above result was given in \cite{Biet} where the authors proved that if the Kolmogorov width has polynomial decay with rate $n^{-s}$, then the greedy algorithm also yields  the same rate, i.e., $
\sigma_n(\mathcal{K})_H\leq Cn^{-s}.
$
The estimate of this type was extended  for general Banach spaces $X$ in \cite{DePeVo},  but there is an additional factor $n^{\frac{1}{2}+\varepsilon}$ (for any $\varepsilon>0$) in the rate of decay of $\sigma_n(\mathcal{K})_X$ compared to that of $d_n(\mathcal{K})_X$, that is, 
\be\label{in-0}
\sigma_n(\mathcal{K})_X\leq C n^{-s+\frac{1}{2}+\varepsilon}
\ee
where $d_n(\mathcal{K})_X \leq C_0n^{-s}$ and $C$  depends on $s$, $\varepsilon$. 

\smallskip
For a recent  result in this direction we refer to \cite{Woj} where the author attempted to reduce the loss  $n^{\frac{1}{2}+\varepsilon}$. Let $\tilde{\gamma}_n(X)$ be the supremum of Banach-Mazur distance $d(V_n,\ell_2^n)$, where $V_n$ is any  $n$-dimensional subspace of any quotient space of $X$. If $d_n(\mathcal{K})_X\leq C_0n^{-s}$ and  $\tilde{\gamma}_n(X)\leq C_1n^{\mu}$, then Wojtaszczyk \cite{Woj}  shows that there is a constant $C$ such that
\be \label{in-1}
\sigma_n(\mathcal{K})_X\leq C \bigg(\frac{\log (n+2)}{n}\bigg)^{s} n^{\mu}\,.
\ee  
Observe that the estimate given in \eqref{in-1} improves the result  \eqref{in-0} since $\tilde{\gamma}_n(X)\leq \sqrt{n}$. It has been shown in \cite{Woj} that the above estimate is  optimal  in $L_p$ up to a logarithmic factor. 


\smallskip 
In the present paper we will show that the condition on $\tilde{\gamma}(X)$ can be replaced by  $ \gamma_n(X)=\sup_{V_n} d(V_n,\ell_2^n)$ where the supremum is taken over all $n$-dimensional subspaces $V_n$ in $X$, see Section \ref{sec-2} for the definition. In addition, we improve the power of the logarithmic term in \eqref{in-1}   when $s>1/2$. More precisely, we  prove that there is a constant $C>0$ such that
\beqq 
\sigma_n(\mathcal{K})_X\leq  C n^{-s+\mu}\big(\log(n+2)\big)^{\min(s,1/2)}\,, \qquad \text{for}\ \ n\geq 1\,, 
\eeqq
if  $d_n(\mathcal{K})_X\leq C_1n^{-s}$ and $\gamma_n(X)\leq C_2n^{\mu}$\,.

\smallskip
 Often, the compact set of interest $\mathcal{K}$  is the image (or subset) of the closed unit ball $B_E$ of a Banach space $E$ under a compact operator $T\in \mathcal{L}(E,X)$.  For this reason, we shall compare $\sigma_n(\mathcal{K})_X$ with the Kolmogorov widths $d_n(T(B_E))_X$. In this study, we obtain the estimate  
  \be\label{in-3}
  \sigma_{3n-1}(\mathcal{K})_X \leq 3e^2\, \Gamma_n(E)\Gamma_n(X) \Bigg(\prod_{k=0}^{n-1}  d_{k}(T(B_E))_X  \Bigg)^{1/n}  ,\qquad n\geq 1,
  \ee 
where $\Gamma_n(X)$ is the $n$-Grothendieck number of $X$, which is closely related to $\gamma_n(X)$, see Section \ref{sec-2}. Note that if $E$ is a Hilbert space then $\Gamma_n(E)=1$, $n\in \N$. In Section \ref{sec-3} we will present  examples showing that \eqref{in-3} gives the optimal estimate
for the decay of $\sigma_n(\mathcal{K})_X$.

\smallskip
 The rest of our paper is organized as follows. In the next Section \ref{sec-2} we will collect some required tools. The main results are stated and proved in Section \ref{sec-3}.

\section{Some preparations}\label{sec-2}
In this section we collect some  tools  needed to formulate our results in the next section where Theorem \ref{main-1}  follows closely  \cite[Proposition 2.2]{Woj}. The main idea in proving \cite[Proposition 2.2]{Woj} is that the author translated from estimate $\sigma_n(\mathcal{K})_X=\dist (f_n,V_{n})_X$ in the original norm to the estimate  $\dist (f_n,V_{n})$ in an appropriately chosen Euclidean norm on $\linspan\{V_n,f_n\}$ where the author took full advantage of the orthogonality. The cost of this translation is proportional to the Banach - Mazur distance $d(V_n,\ell_2^n)$. In the next step, the author applied the Hadamard's inequality, and then arithmetic-geometric mean inequality as in \cite{DePeVo} to get the bound for $\sigma_n(\mathcal{K})_X$ in terms of $d_n(\mathcal{K})_X$ and $d(V_n,\ell_2^n)$. This estimate is then passed to quotient spaces to obtain the final result given in \eqref{in-1} where the decay of $\tilde{\gamma}_n(X)$ comes into play. 

\smallskip
 The Banach - Mazur distance of two isomorphic Banach spaces $X$ and $Y$ is
defined by
\beqq
d(X,Y)=\inf\big\{\|T\|\cdot \|T^{-1}\|: \ T: X\to Y \text{ is an isomorphism}\big\} \,.
\eeqq
For a Banach space $X$ we introduce a sequence of numbers
\beqq
\gamma_n(X)=\sup\big\{d(V,\ell_2^n)\,,\ V \ \text{is an } n \text{-dimensional subspace in }X \big\}\,.
\eeqq
The sequence $\{\gamma_n(X)\}_{n\geq 1}$ is non-decreasing and $\gamma_1(X)=1$. It is obvious that if $X$ is a Hilbert space then we have $\gamma_n(X)=1$, $n=1,2,3,\ldots$. In the case of an arbitrary Banach space $X$, it is known that $\gamma_n(X)\leq n^{1/2}$ and $\gamma_n(L_p)\leq n^{|\frac{1}{2}-\frac{1}{p}|}$ for $1\leq p\leq \infty$. We also define a related notion
\beqq
\tilde{\gamma}_n(X)= \sup\big\{\gamma_n(Z): \ Z \text{{ is a quotient space of }}X \big\}\,.
\eeqq
It is clear that $\gamma_n(X)\leq \tilde{\gamma}_n(X)$. Later we will see in the proof of Theorem \ref{thm1} that $\gamma_n(X)\leq C_1 n^{\mu}$ implies $\tilde{\gamma}_n(X)\leq C_2 n^{\mu}$.

\smallskip
Let $X$ and $Y$ be Banach spaces of finite dimension. Then there exists an operator $T: X\to Y$ such that $d(X,Y)=\|T\|\cdot \|T^{-1}\|$. We can additionally assume that $\|T^{-1}\|=1$. Hence a new norm on $X$ defined by $\|x\|_e:=\|Tx\|_Y$ satisfies
\be \label{dist}
\|x\|_X \leq \|x\|_e\leq d(X,Y)\|x\|_X\,.
\ee 
Moreover $T$ is an isometry between $(X,\|\cdot\|_e)$ and $Y$\,.

\smallskip
The local injective distance $\gamma_n(X)$ is closely related to the  so-called Grothendieck numbers. Let $T\in \mathcal{L}(X,Y)$ be a linear bounded operator. The $n$-th Grothendieck number of $T$ is defined as
\beqq
\Gamma_n(T):=\sup \Big\{ \big|\det\big( \langle Tx_i,b_j\rangle\big)\big|^{1/n},\ x_1,\ldots,x_n\in B_X,\ b_1,\ldots, b_n \in B_{Y'} \Big\}\,.
\eeqq
If $T$ is the identity map of $X$ then we write $\Gamma_n(X)$. Let $0\leq \delta \leq 1/2$. A Banach space $X$ is said to be of weak Hilbert type $\delta$ if there exists a constant $C\geq 1$ such that $\Gamma_n(X)\leq C\, n^\delta$ for $n\geq 1.$ We denote the class of these spaces by $\Gamma_\delta$. Note that $
\Gamma_{1/2}$ is the set of all Banach spaces, i.e.,
\beqq
\Gamma_n(X) \leq C\, n^{1/2},\qquad \text{for all } X\,.
\eeqq
In particular we have
$\Gamma_n(L_p)\leq n^{|1/p-1/2|}$ for $1\leq p\leq \infty$. For further properties of this quantity we refer to \cite{Gei90,Pie91}. The relation between $\gamma_n(X)$ and Grothendieck numbers is represented in the following lemma, see, e.g., \cite{Pie91}.
\begin{lemma}\label{gamma} Let $X$ be a Banach space. Then
$$\gamma_n(X)\leq C_1n^{\delta}\qquad \text{ if and only if }\qquad \Gamma_n(X)\leq C_2n^{\delta}$$ for some $C_1, C_2\geq 1$. 
\end{lemma}

 For later use, let us introduce the notion of Kolmogorov and Gelfand widths of linear continuous operators. In the following we use the definition given in \cite[Chapter 2]{Pin}, but see also \cite[Chapter 11]{Pie}. Note that there is a shift of 1 between definitions by Pinkus \cite[Chapter 2]{Pin} and Pietsch \cite[Chapter 11]{Pie}. Let $B_X$ be the closed unit ball of $X$. The Kolmogorov $n$-width of the operator $T \in \mathcal L(X,Y)$ is defined as 
 \beqq
 d_n(T):=d_n\big(T(B_X)\big)_Y= \inf_{L_n}\sup_{\|x\|_X\leq 1}\inf_{y\in L_n}\|Tx-y\|_Y,  
 \eeqq
 where the infimum is taken over all subspaces $L_n$ of dimension at most $n$ in $Y$. The Gelfand $n$-th width of  $T \in \mathcal L(X,Y)$ is given by 
 $$ d^n(T) := d^n\big(T(B_X)\big)_Y:= \inf_{L^n}\sup_{\|x\|_X\leq 1\,,x\in L^n}\|Tx\|_Y \,,$$
 where the infimum is taken over subspaces $L^n$ of $X$ of co-dimension at most $n$. We also put
 $
 d_0(T)=d^0(T)=\|T\|
 $. Note that Kolmogorov and Gelfand widths are closely related, i.e.,
 $d^n(T)=d_n(T')$ for every $T\in \mathcal{L}(X,Y)$ and $d_n(T)=d^n(T')$ if $T$ is compact or $Y$ is a reflexive Banach space, see \cite[Chapter 2]{Pin}. Therein $T'$ denotes the dual operator of $T$. For basic properties of these quantities we refer to the monographs \cite[Chapters 2]{Pin} and \cite[Chapter 11]{Pie}\,.

\smallskip
The relations between Grothendieck numbers and Kolmogorov or Gelfand widths are given in the following lemma. For a proof we refer to \cite{Pie91}.
\begin{lemma}\label{lem1} Let $X$ and $Y$ be Banach spaces and $T\in \mathcal{L}(X,Y)$. Then it holds
	\beqq
	 \Bigg(\prod_{k=0}^{n-1} d_k(T)\Bigg)^{1/n}\leq \Gamma_{n}(T)\qquad \text{and} \qquad  \Bigg(\prod_{k=0}^{n-1} d^k(T)\Bigg)^{1/n}\leq \Gamma_{n}(T)
	\eeqq
for all $n\geq 1$. 
\end{lemma}
An operator $T\in \mathcal{L}(X,Y)$ is called absolutely $2$-summing if there exists a constant $C$ such that
\be\label{B2}
\Bigg(\sum_{i=1}^n \Vert Tx_i \Vert^2 \Bigg)^{1/2}\leq  C\,   \sup\Bigg\{ \Bigg( \sum_{i=1}^n |\langle x_i,b\rangle|^2\Bigg)^{1/2} : \ b\in X',\ \Vert b \Vert_{X'}\leq 1\Bigg\}.
\ee
The set of these operators is denoted by $\mathcal{B}_{2}(X,Y) $ and the norm $ \|T|\mathcal{B}_2\| $  is given by the infimum of all $C > 0$ satisfying \eqref{B2}. The following assertion can be found in \cite{Pie91}.

\begin{lemma}\label{lem2} Let $X$ and $Y$ be Banach spaces. Let $T\in \mathcal{B}_2(X,Y)$. Then we have
	\beqq
	\Gamma_n(T) \leq en^{-1/2} \|T|\mathcal{B}_2\| \, \Gamma_n(X)\,,\qquad n\geq 1.
	\eeqq
\end{lemma}
\section{Main results}\label{sec-3}
We first show that the assumption $ \tilde{\gamma}_n(X)\leq C_1 n^{\mu}$ in \cite[Theorem 2.3]{Woj} can be replaced by $\gamma_n(X)\leq C_1 n^{\mu}$ as a consequence of Lemma \ref{gamma}. 

\begin{theorem}\label{thm1}
Let $\mathcal{K}$ be a compact set contained in the unit ball of a Banach space $X$. Assume  $$\gamma_n(X)\leq C_1 n^{\mu}\qquad\text{and}\qquad d_n(\mathcal{K})_X\leq C_2 n^{-s},\qquad n\geq 1$$
with  $0< \mu\leq \frac{1}{2}$ and $s> \mu$. 
Then, there exists a constant $C>0$ such that
\beqq 
\sigma_n(\mathcal{K})_X\leq  C n^{-s+\mu}\big(\log(n+2)\big)^s\,, \qquad \ \ n\geq 1\,.
\eeqq
\end{theorem}
\begin{proof}  
 From the assumption $ \gamma_n(X)\leq C_1 n^{\mu}$ and Lemma \ref{gamma} we obtain  $
\Gamma_n(X)=\Gamma_n(X')  \leq C_3 n^\mu$ and $\gamma_n(X')\leq C_4 n^\mu$. For the fact  $\Gamma_n(X)=\Gamma_n(X')$, in particular when $X$ is not reflexive, we refer to \cite[Proposition 1.1]{Gei90}. Let $V$ be any closed subspace in $X$. Then $X/ V$ is isometrically isomorphic with
\beqq
V^{\perp}:=\big\{ f \in X': \ f(v)=0,\ \forall v\in V \big\}\,.
\eeqq
Consequently, we get
\beqq
\gamma_n(X/ V ) = \gamma_n(V^{\perp}) \leq \gamma_n(X')\leq C_4 n^\mu\,,
\eeqq
which implies $\tilde{\gamma}_n(X)\leq C_4 n^\mu$. The proof is completed by \cite[Theorem 2.3]{Woj}\,.
\end{proof}
From the proof of Theorem \ref{thm1} and the definition of $\tilde{\gamma}_n$ we find that for a Banach space $X$, $\gamma_n(X)\leq C_1n^{\mu}$ if and only if $\tilde{\gamma}_n(X)\leq C_2n^{\mu}$  for some $C_1, C_2>0$. In the following theorem we obtain a better   logarithmic power than in  Theorem \ref{thm1} when $s> 1/2$. 

 \begin{theorem}\label{main-1} Let  $X$ be a Banach space and $\mathcal{K}$ a compact subset of $X$. Assume that
 	$$ d_n(\mathcal{K})_X\leq C_0\max(1,n)^{-s}, \ \ (n\geq 0)\qquad\text{and}\qquad \gamma_n(X)\leq C_1 n^{\mu},\ \ (n\geq 1)\,$$
with  $0\leq \mu\leq \frac{1}{2}$ and $s> \mu$.  Then we have 
 \be \label{k-10}
 \sigma_n(\mathcal{K})_X\leq  C_0 C_1 2^\mu 16^s\sqrt{\log (2n)}\, n^{-s+\mu}\, \qquad \text{for}\ \ n\geq 2\,.
 \ee 

 \end{theorem}
 \begin{proof} Without loss of generality we assume
 that $X$ is an infinite-dimensional Banach space, since otherwise the claim follows trivially.
 \\
 {\it Step 1.} 	Let $\varepsilon>0$. From the assumption
	 	$d_n(\mathcal{K})_X\leq C_0n^{-s}$, $n\geq 1$,
	 	we infer the existence of a sequence of subspaces $(S_k)_{k\geq 0}$ in $X$ and $\dim (S_k)=2^{k}$ such that
	 	\beqq 
	 	 \max_{x\in \mathcal{K}}\min_{g\in S_k}\|x-g\|_X\leq (C_0+\varepsilon)2^{-sk}\,.
	 	\eeqq 
	 For $n\in \N$ fixed we put $V_k=S_0+S_1+\ldots+S_{k-1}$ for $k=1,\ldots,n$. Then we have $V_{k}\subset V_{k+1}$ and $\dim(V_k)< 2^{k}$. Observe that
	 \be \label{kol}
	 	 	 \max_{x\in \mathcal{K}}\min_{g\in V_k}\|x-g\|_X\leq 	 	 \max_{x\in \mathcal{K}}\min_{g\in S_{k-1}}\|x-g\|_X\leq (C_0+\varepsilon)  2^{-s(k-1)}\,.
	 \ee

	 We denote $N=2^{n}$. Implementing the greedy algorithm for the set $\mathcal{K}$ we get the sequence $\{f_0,\ldots,f_{N-1}\}$. We assume that the system $\{f_0,\ldots,f_{N-1}\}$ is linearly independent otherwise $\sigma_j(\mathcal{K})_X=0$ for all $j\geq N$ and \eqref{k-10} is obtained. It then follows from \eqref{kol} that
	 \be \label{k-01}
          \|f_\ell-g_\ell^k\|_X \leq (C_0+\varepsilon)  2^{-s(k-1)}\,, \qquad \ell=0,\ldots, N-1\,;\ k=1,\ldots,n
	 \ee 
	 for some $g_\ell^k\in V_k$. Let $Z=\linspan\{f_0,\ldots,f_{N-1} \}$ and $Y=\linspan\{V_n,Z \} $. It is obvious that $2^n\leq \dim(Y)< 2^{n+1}$. From \eqref{dist} we infer the existence of a Euclidean norm $\|\cdot\|_e$ on $Y$ satisfying 
	 \be \label{k-02}
	 \|y\|_X \leq \|y\|_e\leq d\big(Y,\ell_2^{\dim(Y)}\big)\|y\|_X\leq \gamma_{\dim(Y)}(X)\|y\|_X \leq A \|y\|_X,
	 \ee
	  where we put $A= \gamma_{2^{n+1}}(X)$.

\smallskip	 
	  Let $Q$ be the orthogonal projection from $Y$ onto $Z$ in the Euclidean norm $\|\cdot\|_e$. We denote $\dim(Q(V_k))=h_k$ for $k=1,\ldots,n$. It is clear that $h_k\leq \dim(V_k)< 2^k$ and $ Q(V_{k-1})\subset Q(V_k)$.\\
	 From \eqref{k-01} and \eqref{k-02} we get
	 \be \label{k-03}
	 \begin{split} 
	\dist\big(f_\ell,Q(V_k)\big)_{\|\cdot\|_e}&\leq  \|f_\ell-Q(g_\ell^k)\|_e \\
	& = \|Q(f_\ell - g_\ell^k)\|_e \leq \|f_\ell-g_\ell^k\|_e\leq (C_0 +\varepsilon) A {2^{-s(k-1)}}\,.
\end{split}
	 \ee 
	 
\smallskip
	 By $\{\phi_j\}_{j=0,\ldots,N-1}$ we denote the orthonormal system obtained from $f_0,\ldots,f_{N-1}$ by Gram-Schmidt orthogonalization   in the norm $\|\cdot\|_e$. It follows that the matrix $[\langle\phi_j,f_\ell\rangle]_{j,\ell=0}^{N-1}$ has a triangular form, where $\langle \cdot, \cdot\rangle$ denotes scalar product in $(Y,\|\cdot\|_e)$. In particular, on the diagonal we have
	 \be \label{d}
	 \dist\big(f_\ell,\linspan\{ f_0,\ldots,f_{\ell-1} \}\big)_{\|\cdot\|_e} \geq \dist \big(f_\ell,\linspan\{ f_0,\ldots,f_{\ell-1} \}\big)_{X} =\sigma_\ell(\mathcal{K})_X\,.
	 \ee
	  {\it Step 2.}  We consider the case
	  \beqq
	  0<h_{m_1}=\ldots=h_{m_{2}-1}<h_{m_2}=\ldots=h_{m_3-1}<\ldots<h_{m_{L}}=\ldots=h_n
	  \eeqq
	 where $m_1=1$, $m_{L+1}=n+1$.  We denote $\{x_j\}_{j=0,\ldots,N-1}$ another orthonormal basis in $Z$, such that 
		\beqq Q(V_{m_{i-1}})=\ldots=Q(V_{m_{i}-1})=\linspan\big\{ x_{0},\ldots, x_{h_{m_{i-1}}-1}\big\}\,
	 \eeqq
	for $i=2,\ldots,L$. Considering the vector $[\langle x_j,f_\ell\rangle ]_{j=0}^{N-1}$ we observe that
	 \be \label{k-04}
	 \sum_{j=h_{m_{L}}}^{N-1}|\langle x_j,f_\ell\rangle |^2 =\dist\big(f_\ell;Q(V_n)\big)_{\|\cdot \|_e}^2
	 \ee 
	 and 
	 \be \label{k-05}
		 |\langle x_0,f_\ell\rangle |^2\leq \|f_\ell\|_e^2\,,\qquad \quad	 	 \sum_{j=h_{m_{i-1}}}^{h_{m_{i}}-1}|\langle x_j,f_\ell\rangle |^2 \leq \dist\big(f_\ell;Q(V_{{m_{i}}-1})\big)_{\|\cdot \|_e}^2\,,
	 \ee 
for $i=2,\ldots, L$. Note that 
\beqq
\prod_{j=0}^{N-1}\sigma_j(\mathcal{K})_X\leq \prod_{j=0}^{N-1}|\langle\phi_j,f_j\rangle | =\big|\det[\langle \phi_j,f_\ell\rangle ]\big|= \big| \det[\langle x_j,f_\ell\rangle ] \big|\,,
\eeqq
see \eqref{d}. By $k_j$ we denote the $j$-th column of the matrix $ [\langle x_j,f_\ell\rangle  ]_{j,\ell=0}^{N-1}$. Applying Hadamard's inequality and then arithmetic-geometric mean inequality we obtain
	 \be \label{com}
	 \begin{split}
	\Bigg( \prod_{j=0}^{N-1}\sigma_j(\mathcal{K})_X \Bigg)^2&\leq		\big(\det[\langle x_j,f_\ell\rangle ]\big)^2 \\
	& \leq \Bigg( \prod_{j=0}^{h_1-1}\|k_j\|_e^2 \Bigg)\Bigg( \prod_{j=h_{m_L}}^{N-1} \|k_j\|_e^2 \Bigg)\Bigg( \prod_{i=2}^{L} \prod_{j=h_{m_{i-1}}}^{h_{m_{i}}-1} \|k_j\|_e^2 \Bigg) \\ 
			& 
			\leq \Bigg(\frac{1}{h_1}\sum_{j=0}^{h_1-1}\|k_j\|_e^2 \Bigg)^{h_1}\Bigg(\frac{1}{N-h_{m_L}} \sum_{j=h_{m_L}}^{N-1} \|k_j\|_e^2\Bigg)^{N-h_{m_L}}\\
			&\ \quad\times\ \ \prod_{i=2}^{L} \Bigg(\frac{1}{h_{m_{i}}-h_{m_{i-1}}} \sum_{j=h_{m_{i-1}}}^{h_{m_{i}}-1} \|k_j\|_e^2\Bigg)^{h_{m_{i}}-h_{m_{i-1}}} \,.	
	 \end{split}
	 \ee 
From \eqref{k-03}, \eqref{k-04}, and \eqref{k-05} we have
	 \beqq
	  \sum_{j=0}^{h_1-1}\|k_j\|_e^2   =  \sum_{j=0}^{h_1-1}\sum_{\ell=0}^{N-1}|\langle x_j,f_\ell\rangle |^2   \leq \sum_{\ell=0}^{N-1}   \|f_\ell\|_e^2 \leq N A^2 d_0(\mathcal{K})_X^2 \leq NA^2 (C_0+\varepsilon)^2 
	  \eeqq
(since $d_0(\mathcal{K})_X\leq C_0$, by our assumption),	and for $i=2,\ldots,L$,
	  \beqq
	  \begin{split} 
	\sum_{j=h_{m_{i-1}}}^{h_{m_{i}}-1} \|k_j\|_e^2 & =  \sum_{j=h_{m_{i-1}}}^{h_{m_{i}}-1}\sum_{\ell=0}^{N-1} |\langle x_j,f_\ell\rangle |^2 \\
	&\leq  \sum_{\ell=0}^{N-1}\dist\big(f_\ell;Q(V_{m_{i}-1})\big)_{\|\cdot \|_e}^2 \leq N (C_0+\varepsilon)^2   A^2 {2^{-2s(m_{i}-2)}}\,.
\end{split}
	  \eeqq
Similarly, we have
\beqq
\sum_{j=h_{m_L}}^{N-1} \|k_j\|_e^2 \leq  \sum_{\ell=0}^{N-1}\dist\big(f_\ell;Q(V_{n})\big)_{\|\cdot \|_e}^2 \leq N (C_0+\varepsilon)^2   A^2 {2^{-2s(n-1)}}\,.
\eeqq
Inserting this into \eqref{com} we find
	  \be \label{k-31}
	 \begin{split}
	 \Bigg( \prod_{j=0}^{N-1}\sigma_j(\mathcal{K})_X \Bigg)^2 
	  & \leq \Bigg(\frac{NA^2(C_0+\varepsilon)^2}{h_1}\Bigg)^{h_1} \Bigg(\frac{ N (C_0+\varepsilon)^2   A^2 {2^{-2s(n-1)}}}{N-h_{m_L}}\Bigg)^{N-h_{m_L}} \\
	  &\qquad\times\quad\prod_{i=2}^{L}  \Bigg(\frac{N(C_0+\varepsilon)^2   A^2 {2^{-2s(m_{i}-2)}}}{h_{m_{i}}-h_{m_{i-1}}} \Bigg)^{h_{m_{i}}-h_{m_{i-1}}}\\
	 	& = M A^{2N}  (C_0+\varepsilon)^{2N}\Bigg(2^{-2s(n-1)(N-h_{m_L})}\prod_{i=2}^{L}    2^{-2s(m_{i}-2)(h_{m_{i}}-h_{m_{i-1}})}\Bigg)\,,
	 \end{split}
	 \ee 
	 where we put
	 \beqq
	 M= \bigg(\frac{N}{h_1}\bigg)^{h_1}\bigg(\frac{N}{N-h_{m_L}}\bigg)^{N-h_{m_L}}\prod_{i=2}^{L}\bigg(\frac{N}{h_{m_{i}}-h_{m_{i-1}}}\bigg)^{h_{m_{i}}-h_{m_{i-1}}}\,.
	 \eeqq 
For nonnegative number $a_1,\ldots,a_n$ and positive numbers $p_1,\ldots,p_n$ we have
	\beqq
	a_1^{p_1}\cdots a_n^{p_n} \leq \bigg( \frac{a_1p_1+\ldots+a_np_n}{p_1+\ldots+p_n}\bigg)^{p_1+\ldots+p_n}\,,
	\eeqq
see, e.g., \cite[Page 17]{Hardyet}. Applying the above inequality for $M$ we get 
\be\label{k-32}
M\leq (L+1)^{N}\leq (n+1)^{N}\,.
\ee 
Now we deal with the term
\beqq
\begin{split} 
U:&= 2^{-2s(n-1)(N-h_{m_L})}\prod_{i=2}^{L}    2^{-2s(m_{i}-2)(h_{m_{i}}-h_{m_{i-1}})} \\
&=2^{2s[-(n-1)2^n+(n-1)h_{m_L}-(m_{L}-2)(h_{m_L}-h_{m_{L-1}})-\ldots-(m_2-2)(h_{m_2}-h_{m_1})]}\\
&=2^{2s[-(n-1)2^n+h_{m_L}(n+1-m_{L})+h_{m_{L-1}}(m_{L}- m_{L-1})+\ldots + h_{m_2}(m_3-m_2)+h_{m_1}(m_2-2)]} \,.
\end{split}
\eeqq
Using  $h_{m_i}<2^{m_i}$ for $i=1,\ldots, L$ we can estimate
\be\label{k-33}
\begin{split} 
	U
	&\leq 2^{2s[-(n-1)2^n+2^n+\ldots +2^2+ 2^1]}  \leq 2^{2s[-(n-3)2^n]}\,.
	\end{split}
	\ee
Plugging \eqref{k-32} and \eqref{k-33} into \eqref{k-31} we obtain
	  \beqq
	  \begin{split}
	  	\Bigg( \prod_{j=0}^{N-1}\sigma_j(\mathcal{K})_X \Bigg)^2 	  	
	  	& = (n+1)^{N}  A^{2N}(C_0+\varepsilon)^{2N} 2^{[(-n+3)2^n]2s}\,.
	  \end{split}
	  \eeqq
Finally from the assumption $A\leq C_1 2^{\mu(n+1)}$ we find
	 \beqq
	 \begin{split} 
	 \sigma_{2^n-1}(\mathcal{K})_X & \leq (C_0+\varepsilon) C_1 \sqrt{n+1}\cdot 2^{(n+1)\mu}\,2^{(-n+3)s}\\
	 &
	 = (C_0+\varepsilon) C_1 8^s \sqrt{n+1}\cdot 2^{(n+1)\mu}\,2^{-ns}
	\end{split}
	 \eeqq
and hence
\be \label{log}
\sigma_{j}(\mathcal{K})_X \leq (C_0+\varepsilon) C_1 2^\mu 16^s\sqrt{\log(2j)}\cdot j^{(\mu-s)}\, 
\ee 
for $2^{n}\leq j<2^{n+1}$. Since $\varepsilon>0$ arbitrary we get \eqref{k-10}.\\
{\it Step 3.} We comment on the case
	  \beqq
	  0=h_{m_1}=\ldots=h_{m_{2}-1}<h_{m_2}=\ldots=h_{m_3-1}<\ldots<h_{m_{L}}=\ldots=h_n\,.
	  \eeqq
	In this situation we proceed as in Step 2, but there is no first term in the product on the right-hand side of \eqref{com}. Note that in case $h_1=\ldots=h_n=0$, there is no logarithmic factor on the right-hand side of \eqref{log}. The proof is complete.
 \end{proof} 
 Combining Theorems \ref{thm1} and \ref{main-1} we have the following. 
 \begin{cor}  Let $\mathcal{K}$ be a compact set contained in the unit ball of a Banach space $X$ such that $ \gamma_n(X)\leq C_1 n^{\mu}$ and 
 	$d_n(\mathcal{K})\leq C_2 n^{-s},\ n\geq 1.$ 	Then, there exists a constant $C>0$ such that
 	\beqq 
 	\sigma_n(\mathcal{K})_X\leq  C n^{-s+\mu}\big(\log(n+2)\big)^{\min(s,1/2)}\, \qquad \text{for}\ \ n\geq 0\,. 
 	\eeqq 
 	In particular, when $X=L_p$ with $1\leq p\leq \infty$ and $s>\big|\frac{1}{2}-\frac{1}{p}\big|$ we have
 	\beqq
 	\sigma_n(\mathcal{K})_{L_p}\leq C  n^{-s+|\frac{1}{2}-\frac{1}{p}|}\big(\log (2+n)\big)^{\min(1/2,s)}\,\, \qquad \text{for}\ \ n\geq 0\,.
 	\eeqq
 \end{cor}

Let $E$ and $X$ be Banach spaces and $B_E$ be the closed unit ball of $E$. We proceed  to the consideration of the case   $\mathcal{K}\subset  T(B_E)$  where $T\in \mathcal{L}(E,X)$ is a compact operator. We shall compare the rate of convergence of $\sigma_n(\mathcal{K})_X$ with the Kolmogorov widths of $T(B_E)$.

 \begin{theorem}\label{greedy2} Let $X$ be a Banach space and $\mathcal{K}$ be a compact set in $X$. Assume that there exists a compact operator $T\in \mathcal{L}(E,X)$ such that $\mathcal{K}\subset T(B_E)$, where $E$ is a reflexive Banach space. Then we have
 	\beqq
 	\Bigg(\prod_{k=0}^{3n-1} \sigma_k(\mathcal{K})_X\Bigg)^{1/3n} \leq 3e^2 \Gamma_n(E) \Gamma_n(X) \Bigg(\prod_{k=0}^{n-1}  d_{k}(T)  \Bigg)^{1/n}  ,\qquad n\geq 1.
 	\eeqq
 \end{theorem}
 \begin{proof} First, note that $T(B_E)$ is a closed set in $X$ since $T$ is a compact operator and $E$ is reflexive. For $n\in \N$ fixed, running the greedy algorithm for $\mathcal{K}$ we get $\{f_0,\ldots,f_{3n-1}\}$ and $V_{k}=\linspan\{f_0,\ldots,f_{k-1} \}$. 
 	We select $e_k\in B_E$ such that $Te_k=f_k$ for $k=0,\ldots, 3n-1$. 
 	For each $k\in \N$,	as a consequence of the Hahn-Banach Theorem, see \cite[Corollary 14.13]{HeStr}, we can choose $b_k\in X'$ such that $\|b_k\|_{X'}=1$,
 	\be \label{key}
 	\langle  V_k,b_k\rangle =0\,,\qquad \text{and} \qquad \langle  f_k,b_k\rangle=\langle  Te_k,b_k\rangle=\dist(f_k,V_k)_X=\sigma_k(\mathcal{K})_X\,.
 	\ee 
 We define the operators $A\in \mathcal{L}(\ell_2^{3n},E)$ and $B\in \mathcal{L}(X,\ell_2^{3n})$ by
 	\be \label{b}
 	A:=\sum_{k=0}^{3n-1} u_k\otimes e_k\qquad \text{and} \qquad B:=\sum_{k=0}^{3n-1}b_k\otimes u_k\,,
 	\ee 
 	where $ \{u_k\}_{k=0,\ldots,3n-1}$ is the canonical basis of $\ell_2^{3n}$. We calculate the norm $\|\,B\,|\mathcal{B}_2\|$, see the definition \eqref{B2}. Let $x_1,\ldots, x_N \in X$. We have
 	\beqq
 	\Bigg(\sum_{i=1}^N \| Bx_i \|_{\ell_2^{3n}}^2 \Bigg)^{1/2}= \Bigg(\sum_{i=1}^N \Big\| \sum_{k=0}^{3n-1} \langle x_i,b_k \rangle  u_k \Big\|_{\ell_2^{3n}}^2 \Bigg)^{1/2} =  \Bigg(\sum_{i=1}^N   \sum_{k=0}^{3n-1} |\langle x_i,b_k \rangle   |^2 \Bigg)^{1/2}\,
 	\eeqq
 	which implies 
 	\beqq
 	\begin{split} 
 	\Bigg(\sum_{i=1}^N \| Bx_i \|_{\ell_2}^2 \Bigg)^{1/2} & \leq \sqrt{3n}\,\sup_{k=0,\ldots,3n-1}  \Bigg(\sum_{i=1}^N     |\langle x_i,b_k \rangle   |^2 \Bigg)^{1/2}
 	\\
 	& \leq \sqrt{3n}\sup_{b\in B_{X'}}\Bigg(\sum_{i=1}^N     |\langle x_i,b \rangle   |^2 \Bigg)^{1/2}\,.
 \end{split}
 	\eeqq
 	Hence $\|\,B\,|\mathcal{B}_2 \|\leq \sqrt{3n} $\,. 
 	We consider the matrix
 	$
 	( \langle Te_k,b_j \rangle) = ( \langle BTAu_k,u_j \rangle)
 	$
 	which is lower triangular by the choice of the functional $b_k$. It follows from \eqref{key} that
 	\beqq
 \Bigg(\prod_{k=0}^{3n-1} \sigma_k(\mathcal{K})_X\Bigg)^{1/3n}=\Bigg(\prod_{k=0}^{3n-1} | \langle Te_k,b_k\rangle|\Bigg)^{1/3n}   = \big|\det \big( \langle BTAu_i,u_j \rangle\big) \big|^{1/3n}\,.
 \eeqq
 	Note that for any operator $S\in \mathcal{L}(\ell_2^n)$ we have
 	$
 	|\det S| \leq \prod_{k=0}^{n-1} d_k(S)\,
 	$
	since Kolmogorov widths are equal to singular values of $S$, see \cite{Pie2}. Consequently we obtain 
 	\be \label{sim}
 	\begin{split} 
 		\Bigg(\prod_{k=0}^{3n-1} \sigma_k(\mathcal{K})_X\Bigg)^{1/3n} & \leq \Bigg(\prod_{k=0}^{3n-1}  d_k(BTA) \Bigg)^{1/3n} \\
 		& =  \Bigg(\prod_{k=0}^{n-1}  d_{3k}(BTA) \prod_{k=0}^{n-1}  d_{3k+1}(BTA)\prod_{k=0}^{n-1}  d_{3k+2}(BTA) \Bigg)^{1/3n} 
 		\,.
 	\end{split}
 	\ee 
 	From the property
 	\beqq
 	d_{m+n+k}(BTA)\leq d_m(B)d_n(T)d_k(A), 
 	\eeqq
see \cite[Page 32]{Pin} or \cite[Theorem 11.9.2]{Pie}, and the monotonicity $d_{k+1}\leq d_k$ of  Kolmogorov widths we conclude that
 	\be \label{pro}
 	\begin{split} 
 		\Bigg(\prod_{k=0}^{3n-1} \sigma_k(\mathcal{K})_X\Bigg)^{1/3n}  
 		&
 		\leq \Bigg(\prod_{k=0}^{n-1}  d_{k}(A)  \Bigg)^{1/n}\Bigg(\prod_{k=0}^{n-1}  d_{k}(T)  \Bigg)^{1/n}  \Bigg(\prod_{k=0}^{n-1}  d_{k}(B) \Bigg)^{1/n}  \,.
 	\end{split}
 	\ee 
 	Lemmas \ref{lem1} and \ref{lem2} yield the estimate
 	\be \label{B}
 	\begin{split} 
 		 	\Bigg(\prod_{k=0}^{n-1}  d_{k}(B) \Bigg)^{1/n} \leq \Gamma_n(B) & \leq en^{-1/2}\| B|\mathcal{B}_2\|\, \Gamma_n(X)\\
 		 	& \leq e n^{-1/2} (3n)^{1/2} \, \Gamma_n(X)=e\sqrt{3}\,\Gamma_n(X) \,.
 	\end{split}
 	\ee 
 	Now we deal with the first product on the right-hand side of \eqref{pro}. We have
 	\beqq
 	\Bigg(\prod_{k=0}^{n-1}  d_{k}(A)  \Bigg)^{1/n} = \Bigg(\prod_{k=0}^{n-1}  d^{k}(A')  \Bigg)^{1/n} \leq \Gamma_n(A')\,,
 	\eeqq
 see Section \ref{sec-2}. Here $A'\in \mathcal{L}(E',\ell_2^{3n})$ is the dual operator of $A$ which is of the form 
 	$
 	A'=\sum_{k=0}^{3n-1} e_k\otimes u_k\,.
 	$ Similar argument as for the operator $B$ we also get $\|A'|\mathcal{B}_2\| \leq \sqrt{3n} $. 
 	Hence we found 
 	\beqq
 	\begin{split} 
 	\Gamma_n(A') & \leq   en^{-1/2}\| A'|\mathcal{B}_2\|\, \Gamma_n(E') \leq e n^{-1/2} (3n)^{1/2} \, \Gamma_n(E')=e\sqrt{3}\,.\Gamma_n(E')
 \end{split}
 	\eeqq
 	which leads to
 	\beqq
 	\Bigg(\prod_{k=0}^{n-1}  d_{k}(A)  \Bigg)^{1/n} \leq e\sqrt{3}\, \Gamma_n(E')= e\sqrt{3}\, \Gamma_n(E)\,.
 	\eeqq
Putting this and \eqref{B} into \eqref{pro} we arrive at
 	\beqq
 	\begin{split} 
 		\Bigg(\prod_{k=0}^{3n-1} \sigma_k(\mathcal{K})_X\Bigg)^{1/3n}  
 		&
 		\leq 3e^2 \Gamma_n(E)\Gamma_n(X)\Bigg(\prod_{k=0}^{n-1}  d_{k}(T)  \Bigg)^{1/n}      \,.
 	\end{split}
 	\eeqq
The proof is complete.
 \end{proof}
 \begin{remark} Note that Theorem \ref{greedy2} still holds true if one replaces Kolmogorov widths by Gelfand widths, i.e.,
 	 	\beqq
 	 	\Bigg(\prod_{k=0}^{3n-1} \sigma_k(\mathcal{K})_X\Bigg)^{1/3n} \leq 3e^2 \Gamma_n(E) \Gamma_n(X) \Bigg(\prod_{k=0}^{n-1}  d^{k}(T)  \Bigg)^{1/n}  ,\qquad n\geq 1.
 	 	\eeqq
 \end{remark}
  We have the following consequence.
  \begin{cor} \label{cor-5} Let $X$ be a Banach space, $H$ be a Hilbert space, and $\mathcal{K}$ be a compact set in $X$. Assume that there exists a compact operator $T\in \mathcal{L}(H,X)$ such that $\mathcal{K}\subset T(B_H)$. Then we have
  	\be \label{k-11}
  	\sigma_{3n-1}(\mathcal{K})_X \leq 3e^2 \Gamma_n(X) \Bigg(\prod_{k=0}^{n-1}  d_{k}(T)  \Bigg)^{1/n}  ,\qquad n\geq 1.
  	\ee 	
  	In addition, if $X=L_p$ for $1\leq p\leq \infty$ and  $d_n(T) \leq C_0 n^{-s}$ for some $s> \big|\frac{1}{2}-\frac{1}{p}\big|$, $(n\geq 1)$, then there exists a constant $C>0$ such that
  	\beqq
  	\sigma_{n}(\mathcal{K})_X \leq C \, n^{|1/p-1/2|-s}     ,\qquad n\geq 1.
  	\eeqq
  \end{cor}
\begin{proof}
Since the sequence  $\sigma_{n}(\mathcal{K})_X$ is monotonically non-increasing and $\Gamma_{n}(H)=1$, $n\in \N$, we obtain \eqref{k-11} straightforwardly from Theorem \ref{greedy2}. Now, with $X=L_p$, $1\leq p\leq \infty$, and  $d_n(T) \leq C_0 n^{-s}$ we have for $n\geq 1$
   	\beqq
   	\begin{split}
 \sigma_{3n-1}(\mathcal{K})_X &
  \leq 
  3e^2   n^{|1/p-1/2|} \Bigg( \|T\|\prod_{k=1}^{n-1}  C_0 k^{-s} \Bigg)^{1/n}   
 \\
 &  \leq
 3e^2 \max\big\{C_0,\|T\| \big\}   n^{|1/p-1/2|} \big((n-1)! \big)^{-s/n} \,.
   	\end{split}
 \eeqq
Using Stirling's approximation $(n-1)! \geq  \sqrt{2\pi}(n-1)^{n-1/2}e^{-n+1}$, we arrive at
 \beqq
 	\sigma_{3n-1}(\mathcal{K})_X \leq  C n^{|1/p-1/2|-s}  \,,\qquad n\geq 2,
 \eeqq
with  $C>0$ independent of $n$. From this, the second assertion follows.
\end{proof}
 As a supplement we study the case $\mathcal{K}=T(B_{\ell_2})$ for some compact operator $T\in \mathcal{L}(\ell_2,X)$. In this situation we can replace $ \sigma_{3n-1}(\mathcal{K})_X$  in \eqref{k-11} by $ \sigma_{2n-1}(\mathcal{K})_X$. We have the following.
 \begin{theorem} Let $X$ be a Banach space and $T\in \mathcal{L}(\ell_2,X)$ be a compact operator. Assume that $\mathcal{K}= T(B_{\ell_2})$. Then we have
 	\beqq
 	 	\sigma_{2n-1}(\mathcal{K})_X \leq e\sqrt{2} \Gamma_n(X) \Bigg(\prod_{k=0}^{n-1}  d_{k}(\mathcal{K})_X  \Bigg)^{1/n}  ,\qquad n\geq 1.
 	\eeqq
 \end{theorem}
 \begin{proof} Recall that $T(B_{\ell_2})$ is a closed set in $X$.  For $n\in \N$ fixed, running the greedy algorithm for $\mathcal{K}$ we get $\{f_0,\ldots,f_{2n-1}\}$ and $V_{k}=\linspan\{f_0,\ldots,f_{k-1} \}$. First we  show that we can select $e_k\in B_{\ell_2}$ such that $Te_k=f_k$ for $k=0,\ldots, 2n-1$ and  $\{ e_k\}_{k=0}^{2n-1}$ is an orthonormal system in $\ell_2$. Indeed, if $Te_0=f_0$ with 
 	\beqq
 	\|f_0\|_X =\max_{f\in \mathcal{K}}\|f\|_X =\max_{e\in B_{\ell_2}}\|Te\|_X \,,
 	\eeqq
 then $\|e_0\|_{\ell_2}=1$. Assume that we have chosen the orthonormal system $\{e_0,\ldots,e_{k-1} \}$ in $\ell_2$  with $Te_i=f_i$ for $i=0,\ldots,k-1$. Let $\{e_j \}_{j\geq 0}$ be an orthonormal basis of $\ell_2$ constructed from the system $\{e_0,\ldots,e_{k-1} \}$  and let $f_k=Te$ where $e=\sum_{j\geq0} c_je_j$ with $\|(c_j)_{j\geq 0}\|_{\ell_2}\leq 1$. We consider
 \beqq
 f_k^*=\frac{1}{\|c^*\|_{\ell_2}}\sum_{j\geq k} c_jTe_j,\ \qquad \text{with}\ \ c^*=(0,\ldots,0,c_k,c_{k+1},\ldots)\,.
 \eeqq
Here we assume that $c^*\not =0$ otherwise $\sigma_k(\mathcal{K})_X=0$. We have $f_k^*\in T(B_{\ell_2})$ and
 \beqq
 \begin{split} 
\dist(f_k,V_k)&\geq \dist(f_k^*,V_k)_X\\ &=\inf_{a_0,\ldots,a_{k-1}}\Bigg\|\frac{1}{\|c^*\|_{\ell_2}}\sum_{j\geq k} c_jTe_j - \sum_{j=0}^{k-1}a_j Te_j\Bigg\|_X\\
 & = \frac{1}{\|c^*\|_{\ell_2}}\inf_{a_0,\ldots,a_{k-1}}\Bigg\| \sum_{j\geq 0} c_jTe_j - \sum_{j=0}^{k-1}\big(a_j\|c^*\|_{\ell_2}+c_j\big) Te_j\Bigg\|_X \\
 & = \frac{1}{\|c^*\|_{\ell_2}} \dist(f_k,V_k)\,.
\end{split}
 \eeqq
 Hence $\|c^*\|_{\ell_2}=1$ and $f_k=f_k^*$ which implies $e$ is orthogonal to $\{e_0,\ldots,e_{k-1} \}$ and $\|e\|_{\ell_2}=1$. Similar to \eqref{b} we define the operators $A\in \mathcal{L}(\ell_2^{2n},\ell_2)$ and $B\in \mathcal{L}(X,\ell_2^{2n})$ by
 \beqq
 A:=\sum_{k=0}^{2n-1} u_k\otimes e_k\qquad \text{and} \qquad B:=\sum_{k=0}^{2n-1}b_k\otimes u_k\,.
 \eeqq
 Note that $\|B|\mathcal{B}_2\|\leq \sqrt{2n}$ and $\|A\|\leq 1$\,.  We have	\beqq
 \begin{split} 
 	\Bigg(\prod_{k=0}^{2n-1} \sigma_k(\mathcal{K})_X\Bigg)^{1/2n} & \leq \Bigg(\prod_{k=0}^{2n-1}  d_k(BTA) \Bigg)^{1/2n}\\
 	& \leq \Bigg(\prod_{k=0}^{2n-1}  \|A\| d_k(BT) \Bigg)^{1/2n}\leq \Bigg(\prod_{k=0}^{2n-1}   d_k(BT) \Bigg)^{1/2n}\,,
 \end{split}
 \eeqq
see \eqref{sim}. By the same argument as in the proof of Theorem \ref{greedy2} we obtain the desired estimate. 
 \end{proof}

We now present an example which shows that  Theorem \ref{greedy2} can be applied to get a better convergence rate of greedy algorithm compared to Theorems \ref{thm1} and \ref{main-1}. 
Let $D\subset \R^d$ be a bounded Lipschitz domain. For $1< q< \infty$ and $s\geq 0$ we define the Bessel potential space $H^s_q(\R^d)$ as the collection of all tempered distributions $f$ such that
\beqq
\|f\|_{H^s_q(\R^d)}:= \big\| \mathcal{F}^{-1}\big[(1+|\xi|^2)^{s/2}\mathcal{F} f \big] (\cdot) \big\|_{L_q(\R^d)}<\infty,\qquad \xi \in \R^d.
\eeqq
Therein $\mathcal{F}$ and $\mathcal{F}^{-1}$ denote the Fourier transform and its inversion. The space on domain $H^s_q:=H^s_q(D)$ is then defined by restricting $H^s_q(\R^d)$ onto $D$. For $s\in \N$ we have $H^s_q(D)=W^s_q(D)$, the Sobolev space  of all functions $ f\in L_q(D)$ such that $ \partial^{\alpha}f \in L_q(D)$ for all $\alpha=(\alpha_1,\ldots,\alpha_d) \in \N_0^d$ with $\alpha_1+\ldots+\alpha_d\leq s$.  

\smallskip
We consider the compact embedding $id: 
H^s_q   \to L_p
$ with $1<q<\infty$, $1\leq p\leq \infty$, and $\frac{s}{d}>\max(\frac{1}{2},\frac{1}{q})$. 
The order of Kolmogorov widths of $B_{H^s_q}$ in $L_p$ is well-known:
\be \label{dn}
d_n(B_{H_q^s})_{L_p} \asymp   n^{-\frac{s}{d}+ \max\big\{\frac{1}{q}-\max(\frac{1}{2},\frac{1}{p}),0\big\}}\,,
\ee 
see \cite[Chapter 5]{Lub82} or \cite[Page 115]{Tem93}, where for two sequences $a_n$ and $b_n$,  the notion $a_n \asymp b_n$  indicates that there exist positive constants $C_1$ and $C_2$ such that $C_1a_n \leq  b_n\leq C_2b_n$. Choosing $t$ such that $\frac{t}{d}=\frac{s}{d}-\max(\frac{1}{q}-\frac{1}{2},0)$ we get $B_{H^s_q} \subset B_{H_2^t}$ and $H^t_2$ is compactly embedded into $L_p$ since $\frac{t}{d}> \frac{1}{2}$. Now, due to $d_n(B_{H_2^t})_{L_p} \asymp   n^{-\frac{t}{d}}$, Corollary \ref{cor-5} yields
\beqq
\sigma_n(B_{H^s_q})_{L_p} \leq  C n^{-\frac{t}{d} +|\frac{1}{p}-\frac{1}{2}|} = Cn^{-\frac{s}{d}+ \max(\frac{1}{q}-\frac{1}{2},0)+|\frac{1}{p}-\frac{1}{2}|}\,.
\eeqq
Comparing this estimate with the Kolmogorov widths in \eqref{dn} we find
\beqq
\sigma_n(B_{H^s_q})_{L_p} \leq C' n^{|\frac{1}{p}-\frac{1}{2}|} d_n(B_{H_q^s})_{L_p}
\eeqq
when $\max(p,q)\geq 2$. 
\vskip 2mm

 The following example taken from \cite{DePeVo}, see also \cite{Woj}, indicates that the estimate given in Corollary \ref{cor-5} is optimal. Let $\mathcal{K}=\{n^{-\alpha}u_n \}\subset \ell_p$ with $2<p<\infty$ where $\{u_n\}_{n\geq 1}$ is the canonical basis of $\ell_2$.  It is clear that $\sigma_n(\mathcal{K})_{\ell_p}=\frac{1}{(n+1)^{\alpha}}$. We consider the diagonal operator $D_\alpha:\ell_2\to \ell_p$ defined by $u_n\to n^{-\alpha}u_n$. Then $\mathcal{K}\subset  D_\alpha(B_{\ell_2})$. We know that $d_n(D_{\alpha}) \leq Cn^{-\alpha+\frac{1}{p}-\frac{1}{2}}$, see, e.g., \cite[Section 6.2.5.3]{Pie07} which implies 
 \beqq
 \sigma_{3n-1}(\mathcal{K})_{\ell_p} \leq C n^{|\frac{1}{2}-\frac{1}{p}|}
n^{-\alpha+\frac{1}{p}-\frac{1}{2}}=Cn^{-\alpha}\,.
 \eeqq
Hence the operator $D_\alpha$ give the sharp estimate in the rate of convergence in this example.
\vskip 3mm
\noindent
{\bf Acknowledgments} The author would like to thank Markus Bachmayr for  fruitful discussions. Moreover, the author acknowledges the Hausdorff Center of Mathematics, University of Bonn for financial support. Finally, the author is grateful to the referees for a careful reading and many detailed hints to improve the paper.

\bibliographystyle{amsplain}
\bibliography{Greedy}

\providecommand{\bysame}{\leavevmode\hbox to3em{\hrulefill}\thinspace}
\providecommand{\MR}{\relax\ifhmode\unskip\space\fi MR }
\providecommand{\MRhref}[2]{%
  \href{http://www.ams.org/mathscinet-getitem?mr=#1}{#2}
}
\providecommand{\href}[2]{#2}
\begin{thebibliography}{10}

\bibitem{Biet}
P.~Binev, A.~Cohen, W.~Dahmen, R.~DeVore, G.~Petrova, and P.~Wojtaszczyk,
  \emph{{Convergence rates for greedy algorithms in reduced basis methods}},
  SIAM J. Math. Anal. \textbf{43} (2011), 1457--1472.

\bibitem{Buet}
A.~Buffa, Y.~Maday, A.T. Patera, C.~Prud’homme, and G.~Turinici, \emph{{A
  priori convergence of the greedy algorithm for the parametrized reduced
  basis}}, Model. Math. Anal. Numer. \textbf{46} (2012), 595--603.

\bibitem{DePeVo}
R.~DeVore, G.~Petrova, and P.~Wojtaszczyk, \emph{{Greedy algorithms for reduced
  bases in Banach spaces}}, Constr. Approx. \textbf{37} (2013), 455--466.

\bibitem{Gei90}
S.~Geiss, \emph{{Grothedieck numbers of linear and continuous operators on
  Banach spaces}}, Math. Nachr. \textbf{148} (1990), 65--79.

\bibitem{Hardyet}
G.H. Hardy, J.E. Littlewood, and G.~P\'olya, \emph{{Inequalities}}, Cambridge
  University Press, London, 1934.

\bibitem{HeStr}
E.~Hewitt and K.~Stromberg, \emph{{Real and Abstract Analysis}}, Springer,
  Berlin, 1969.

\bibitem{Lub82}
C.~Lubitz, \emph{{Weylzahlen von Diagonaloperatoren und Sobolev-Einbettungen}},
  Ph.D. thesis, Bonner Math. Schriften, 144, Bonn, 1982.

\bibitem{May1}
Y.~Maday, A.~T. Patera, and G.~Turinici, \emph{{A priori convergence theory for
  reduced-basis approximations of single-parametric elliptic partial
  differential equations}}, J. Sci. Comput. \textbf{17} (2002), 437--446.

\bibitem{May2}
\bysame, \emph{{Global a priori convergence theory for reduced-basis
  approximations of single-parameter symmetric coercive elliptic partial
  differential equations}}, C. R. Acad. Sci., Paris, Ser. I, Math \textbf{335}
  (2002), 289--294.

\bibitem{Pie}
A.~Pietsch, \emph{{Operator Ideals}}, North-Holland, Amsterdam, 1980.

\bibitem{Pie2}
\bysame, \emph{{Weyl numbers and eigenvalues of operators in Banach spaces}},
  Math. Ann. \textbf{247} (1980), 149--168.

\bibitem{Pie91}
\bysame, \emph{{Eigenvalue distributions and geometry of Banach space}}, Math.
  Nachr. \textbf{160} (1991), 41--81.

\bibitem{Pie07}
\bysame, \emph{{History of Banach Spaces and Linear Operators}}, Birkh\"auser,
  Boston, 2007.

\bibitem{Pin}
A.~Pinkus, \emph{{$n$-Widths in Approximation Theory}}, Springer, Berlin, 1985.

\bibitem{Tem93}
V.~N. Temlyakov, \emph{{Approximation of Periodic Functions}}, Computational
  Mathematics and Analysis Series, Nova Science Publishers, Inc., Commack, NY.,
  1993.

\bibitem{Woj}
P.~Wojtaszczyk, \emph{{On greedy algorithm approximating Kolmogorov widths in
  Banach spaces}}, J. Math. Anal. Appl. \textbf{424} (2015), 685--695.

\end{thebibliography}

\end{document}